\documentclass[12pt]{amsart}

\usepackage{amssymb}

\usepackage{enumitem}

\makeatletter
\@namedef{subjclassname@2020}{%
  \textup{2020} Mathematics Subject Classification}
\makeatother

\usepackage[T1]{fontenc}

\newtheorem{theorem}{Theorem}[section]

\newtheorem{lemma}[theorem]{Lemma}

\theoremstyle{definition}

\numberwithin{equation}{section}

\newcommand{\C}{{\mathbb C}}

\newcommand{\Z}{{\mathbb Z}}
\newcommand{\Q}{{\mathbb Q}}

\begin{document}


\baselineskip=17pt


\title[Correlation and lower bounds]{Correlation and lower bounds of arithmetic expressions}

\author[F. Chamizo]{Fernando Chamizo}
\address{Departamento de Matem\'aticas and ICMAT\\ Universidad Aut\'onoma de Madrid\\
28049 Madrid, Spain}
\email{fernando.chamizo@uam.es}

\date{May 11, 2023}

\begin{abstract}
We explore the use of correlation with simple functions to get lower bounds for arithmetic quantities. In particular, we apply this idea to the power moments of the error term when counting visible lattice points in large spheres. 
\end{abstract}



\maketitle

\section{Average and amplification}

Very often in analytic number theory one has to deal with a certain expression $Q$ that, considering some harmonic companions \cite{iwaniec_h}, can be included in a kind of spectral family $\{Q_j\}_{j\in J}$, 
say $Q=Q_{j_0}$. In this situation, average results are usually easier to get than bounds or asymptotic formulas for~$Q$. For instance, variant of Parseval's identity can lead to something of the form 
\[
 \sum_{j\in J}|Q_j|^2\sim F,
\]
which suggests that $|Q_j|^2$ is typically like $F/|J|$  (if $J$ is not finite there is a chance to introduce weights in the summation) and implies the $\Omega$-result stating that at least one of the $|Q_j|^2$ cannot be asymptotically below $F/|J|$. 

Average results do not give good individual estimates because dropping all the terms except $Q$ is wasteful.
The amplification method, developed by H. Iwaniec and collaborators to get a number of conspicuous results (e.g. 
\cite{FrIw},
\cite{DuFrIw},
\cite{iwaniec_L},
\cite{IwSa}), circumvents this problem. 
As a guide for the reader, to our taste \cite{FrIw_n} is its most transparent application and \cite{IwSa} the most impressive (cf. \cite[\S10]{BlHo}, \cite[pp.\,93--100]{allthat}). 
Let us review briefly the simple and powerful schematic idea in a somewhat restricted setting. A family of linear forms $\mathcal{L}_j(\vec{a})=\sum_n a_n \lambda_j(n)$ is introduced in such a way that there is a quantifiable cancellation in $\sum_j |Q_j|^2\lambda_j(n)\overline{\lambda}_j(m)$ allowing to prove a nontrivial average result of the form 
\[
 \sum_j |Q_j|^2 |\mathcal{L}_j(\vec{a})|^2
 \le K^2\|\vec{a}\|^2.
\]
A choice $\vec{a}=\vec{a}_0$ giving a large value of $|\mathcal{L}_{j_0}(\vec{a})|$ \emph{amplifies} the contribution of $|Q_{j_0}|^2$ to the sum showing 
\[
 |Q|=|Q_{j_0}|
 \le
 \frac{K\|\vec{a}_0\|}{|\mathcal{L}_{j_0}(\vec{a}_0)|}.
\]
The amplifier will be stronger if the vectors $\{\lambda_j(n)\}_n$ keep certain quasi-orthogonality for different values of~$j$ and perfect
orthogonality would ideally allow to select a single term.
It guides our intuition to construct good amplifiers but note that there is no need for proving anything in this direction.

\

Dealing with a specific lattice point problem,
in \cite{ChCrUb} it was introduced a method  that bears a point of resemblance to the amplification philosophy but  producing lower bounds instead of upper bounds. Let us say that our spectral family is now a continuous one represented by an oscillatory function $Q(t)$ that is too complicate, so that one cannot determine the asymptotics of  $\int |Q|^2\, d\mu$. If we find a simpler function $g$, which could be called a \emph{resonator}, in such a way that there is a provable correlation 
\[
 \int Qg\, d\mu \ge F>0
\]
then H\"older's inequality gives 
\[
 \|Q\|_p\ge \frac{F}{\| g\|_q}
 \qquad\text{with}\quad \frac{1}{p}+\frac{1}{q}=1,
 \quad p,q\in [1,\infty].
\]
So, lower bounds for the moments of $Q$ follow from upper bounds for the moments of the simpler function~$g$. 
If $\mu$ is normalized as a probability measure, $\|Q\|_\infty\ge \|Q\|_p$ and an $\Omega$-result for $Q$ is deduced. 
Ideally, we look for~$g$ being a simple proxy of $\overline{Q}$ amplifying the contribution to the integral via some kind of positivity. In contrast, in the amplification method the artificial squared form tries to mimic the delta symbol. 

\

The purpose of this paper is to show some instances of this idea. The main one is a modest improvement on \cite{ChCrUb} in which the resonator was based on $\sum_{d^2\mid n}\mu(d)=\mu^2(n)\ge 0$.  

\section{Main results}

As usual, the notation $e(x)$ abbreviates $e^{2\pi ix}$. 
We also employ $A\ll B$ meaning $A=O(|B|)$ with Landau's notation and $A=\Omega(B)$ as the opposite of $A=o(B)$.

\medskip

We start pushing the correlation technique to improve a little our knowledge about the lattice point problem considered in \cite{ChCrUb}. 

Recall that the \emph{visible points} in $\Z^3-\{\vec{0}\}$ are those having coprime coordinates. If $\mathcal{N}^*(R)$ is the number of them in the ball $\|\vec{x}\|\le R$, it is fairly easy to get $\mathcal{N}^*(R)\sim \frac{4\pi}{3\zeta(3)}R^3$. 
The main result in \cite{ChCrUb} assures
\begin{equation}\label{2moment}
 \int_R^{2R}|E^*|^2\gg R^3\log R
 \qquad\text{where}\qquad
 E^*(R)
 =
 \mathcal{N}^*(R)- \frac{4\pi}{3\zeta(3)}R^3.
\end{equation}
Conjecturally $E^*(R)=O(R^{1+\epsilon}\big)$ for any $\epsilon>0$ and \eqref{2moment} implies that it is false for $\epsilon=0$ and, in fact, $E^*(R)=\Omega(R\sqrt{\log R})$. Here we treat other moments showing that the logarithmic factors are unavoidable in them. 

\begin{theorem}\label{visible}
 For any $p>1$, we have 
 \[
  \int_R^{2R}
  |E^*|^p
  \gg R^{p+1} (\log R)^{p/2}.
 \]
\end{theorem}

Of course, this adds new information to \eqref{2moment} only for $p<2$.

\

The next result has a more analytic flavor.
Essentially, it shows that the existence of gaps in the frequency spectrum implies oscillation. 

\begin{theorem}\label{gaps}
 Let $\{\nu_n\}_{n=1}^N \subset\Z^+$ be strictly increasing with $N<\infty$ or $N=\infty$ and consider its gaps $\Lambda_n = \min_{m\ne n}|\nu_m-\nu_n|$.  
 For each $\alpha\in [0,1]$ and $1<n<N$ the inequality
 \[
  \Big|\sum_{k=1}^N a_k\sin(2\pi \nu_k x)\Big| >
  \frac 14
  \mathcal{B}_{n,\alpha}
  |a_n| |x|^\alpha
 \]
 holds for $x$ in a positive measure subset of  $\big[-\frac 12,\frac 12\big]$, where
  \[
  \mathcal{B}_{n,\alpha}
  =
  \frac{\pi^{\alpha-1}(1-\alpha^2)\Lambda_n}{\Lambda_n^{1-\alpha}-(1/5)^{1-\alpha}}
  \quad\text{for $\alpha\ne 1$},
  \qquad
  \mathcal{B}_{n,1}=\lim_{\alpha\to 1^-}\mathcal{B}_{n,\alpha}
 \]
 and
 $\{a_n\}_{n=1}^N\subset \C-\{0\}$ is any sequence assuring the convergence of the series to an $L^\infty$ function if $N=\infty$.
\end{theorem}

As an arithmetic oriented example, we can deduce from this result that the function 
\[
 F(x)=
 \sum_{n=1}^\infty
 \frac{\tau(n)e(p_n^3 x)}{n^2(\log n)^{2023} },
\]
with $p_n$ the $n$-th prime and $\tau$ the divisor function, does not satisfy the Lipschitz condition at any point. In particular, it is nowhere differentiable.

To see this, consider 
$f(x)=F(x_0+x)-F(x_0-x)$
and note that 
$e\big(a(x_0+x)\big)-e\big(a(x_0-x)\big)$
equals
$2i \, e(ax_0)\sin (2\pi ax)$. 
If $F$ is Lipschitz at $x_0$ then $|f(x)|\le C|x|$
for some constant $C$. When we apply Theorem~\ref{gaps} to $f$ with $\alpha=1$ and $\nu_n=p_n^3$, we have 
$\Lambda_n\ge p_n^2\gg n^2(\log n)^2$ and $\mathcal{B}_{n,1}\gg n^2\log n$. 
Recalling that $\limsup \tau(n)/(\log n)^K=\infty$
\cite[Th.\,314]{HaWr}, the result shows that $|f(x)/x|$ is unbounded.

\

If $F:[1,\infty)\longrightarrow\C$ and $f$ is a bounded arithmetic function, the following kind of convolution is closer to the original formulation of M\"obius inversion and to some natural applications in combinatorics:
\begin{equation}\label{mconv}
 G(x)=
 \sum_{n\le x} f(n)F\Big(\frac xn\Big).
\end{equation}
Let us consider the case in which $F$ is a partial sum of a Fourier series except for introducing a smooth transition from $F(1)$ to $0$ in order to freely extend the upper bound  in \eqref{mconv}. Namely, 
\[
 F(x)=\phi(x)\sum_{n=1}^N a_n e(nx)
\]
with $\phi\in C_0^\infty$ such that $\phi(x)=0$ for $x<1$ and $\phi(x)=1$ for $x>2$.

\begin{theorem}\label{thconv}
 Let $R\ge 4N$ and 
 \[
  \mathcal{B}=\sum_{n=1}^N b_n
  \qquad\text{with}\quad 
  b_n=a_n\sum_{d\mid n} f(d).
 \]
 Assume $b_n\ge 0$ and for some $K>1$ and any $V>1/6$
 \[
  \sum_{\frac R{2V}\le d\le 3R}
  \sum_{\frac{dV}{R}\le r<\frac{2dV}{R}}
  \sum_{\substack{n=1\\ n\equiv \pm r\ (d)}}^N
  |a_n|
  =o\big(
  (6V)^K\mathcal{B}\big). 
 \]
 Then 
 \[
  \int_R^{2R}
  \big|G(x)\big|^2\, dx 
  \gg R\mathcal{B}^2N^{-1}.
 \]
 In particular, $G(x)=\Omega\big(\mathcal{B}N^{-1/2}\big)$.
\end{theorem}

The hypothesis $b_n\ge 0$ is satisfied for instance when $a_n\ge 0$ and $f$ is real character. 
In some sense, the main result in \cite{ChCrUb}, which we refine here, relies on an anharmonic version of this with $a_n\ge 0$ and $f(d)=\mu(\sqrt{d})$ if $d$ is square and $0$ otherwise.

\section{Proofs}

The proof of Theorem~\ref{visible} is based on showing that the asymptotics in \cite[Th.\,1.1]{ChCrUb} is preserved, except for a $\sigma/2$ factor, when replacing $g$ there by the shorter resonator
\[
 g_\sigma(x)
 =
 \sum_{n\le R^\sigma}
 \frac{\cos(2\pi x\sqrt{n})}{\sqrt{n}}
 \qquad\text{where}\quad 0<\sigma<2.
\]
We employ the same notation as in that work introducing 
\[
 I(R)
 =
 \int 
 g_\sigma(t)E^*(t)\, d\nu(t)
\]
where $d\nu(x)=R^{-1}\psi(x/R)$ is a probability measure with $\psi\in C_0^\infty\big((1,2)\big)$ and our goal is

\begin{theorem}\label{shorter}
 Given $0<\sigma<2$, for $R\to\infty$
 \[
  I(R)
  \sim 
  -\sigma C R\log R
 \]
 where
 \[
  C=C_0\int t\psi(t)\, dt
  \quad\text{and}\quad 
  C_0
  =
  \frac 78 
  \prod_{p>2}
  \Big(1-\frac 1p\Big)\Big(1+\frac 1p-\frac{1}{p^2}\Big).
 \]
\end{theorem}

As a matter of fact, in the statement of the main result of \cite{ChCrUb}, formally corresponding to $\sigma\to 2^{-}$, the constant $C_0$ was wrongly substituted by $7/\pi^2$
because $v$ as in the lemma below was not correctly evaluated.
We apologize for the inconvenience. Both constants differ in less than $4\%$ and it does not affect to the $\Omega$-result which was the main interest.

\begin{lemma}\label{v_eval}
 Let $v(d)$ be the number of solutions of $x^2+y^2+z^2\equiv 0\pmod{d^2}$. For $d$ odd squarefree
 \[
  v(2d)=8v(d)
  \qquad\text{and}\qquad 
  v(d)
  =
  \prod_{p\mid d}
  p^2(p^2+p-1).
 \]
 In particular $v(d)=O\big(d^4\log\log d\big)$.
\end{lemma}

\begin{proof}
 By the Chinese remainder theorem, $v$ is multiplicative. Since $n^2\equiv 0,1\pmod{4}$ for every $n$, depending on its parity, $x^2+y^2+z^2\equiv 0\pmod{4}$
 implies $x\equiv y\equiv z\equiv 0\pmod{2}$
 showing $v(2)=2^3$. 
 We have to prove $v(p)=p^2(p^2+p-1)$ for primes $p>2$. Expanding the cube and changing the order of summation, the following exponential sum representation is obtained
 \[
  p^2 v(p)
  =
  \sum_{a=1}^{p^2}
  \bigg(
  \sum_{n=1}^{p^2}
  e\Big(\frac{an^2}{p^2}\Big)
  \bigg)^3.
 \]
 The innermost sum is $p^2$ if $a=p^2$. The classical evaluation of quadratic Gauss sums \cite[(3.38)]{IwKo}
 shows that it is $p$ if $p\nmid a$ and $\Big(\frac kp\Big)c_p$ for certain $|c_p|=\sqrt{p}$ if $a=kp$ with $1\le k<p$. Collecting these contributions, 
 \[
  p^2v(p)
  =
  (p^2)^3+(p^2-p)p^3
  +c_p^3\sum_{k=1}^{p-1}\Big(\frac kp\Big).
 \]
 The sum is null, giving $v(p)=p^2(p^2+p-1)$ as expected. 
 
 The last claim comes from $v(d)<d^4\prod_{p\mid d}(1+p^{-1})$ \cite[\S18.3]{HaWr}. 
\end{proof}

The motivation to take a shorter sum is to control higher moments of~$g_\sigma$. It requires some arithmetic considerations about linear combinations of square roots. 

\begin{lemma}\label{moments}
 For each $k\in\Z^+$ and $0<\sigma<2/\big(2^{2k-1}-1\big)$ there exists a positive constant $C_{k,\sigma}$ such that 
 \[
  \int 
  \Big|
  \sum_{n\le R^\sigma}
  \frac{e(x\sqrt{n})}{\sqrt{n}}
  \Big|^{2k}\, d\nu 
  \sim 
  C_{k,\sigma} (\log R)^k.
 \]
 In particular, $\int |g_\sigma|^{2k}\, d\nu\ll (\log R)^k$.
\end{lemma}

\begin{proof}
 For $\vec{n}\in\Z_{>0}^{2k}$ let 
 $L(\vec{n})
 =
 \sum_{j=1}^k
 \big(\sqrt{n_j}-\sqrt{n_{j+k}}\big)$.
 Opening the power, 
 \[
  \int 
  \Big|
  \sum_{n\le R^\sigma}
  \frac{e(x\sqrt{n})}{\sqrt{n}}
  \Big|^{2k}\, d\nu 
  =
  \sum_{\|\vec{n}\|_\infty \le R^\sigma}
  \frac{\widehat{\psi}\big(RL(\vec{n})\big)}{\sqrt{n_1n_2\cdots n_{2k}}}.
 \]
 By \cite[Lemma\, 2.2]{zhai2}, if $L(\vec{n})\ne 0$ then 
 $\big|L(\vec{n})\big|\gg R^{-\delta}$ with 
 $\delta= \big(2^{2k-1}-1\big)\sigma/2<1$ and
 $\widehat{\psi}\big(RL(\vec{n})\big)\ll R^{-K}$ for any $K>0$, giving a negligible contribution. 
 Hence we have to show 
 \begin{equation}\label{auxm}
  \sum_{\substack{\|\vec{n}\|_\infty \le R^\sigma\\ L(\vec{n})= 0}}
  \frac{1}{\sqrt{n_1n_2\cdots n_{2k}}}
  \sim 
  C_{k,\sigma} (\log R)^k.
 \end{equation}
 It is clear that the terms with 
 $\{n_1,\dots,n_k\}=\{n_{k+1},\dots,n_{2k}\}$
 give the expected asymptotics. We are going to check that the rest of the sum in \eqref{auxm}  is 
 $O\big((\log R)^{k-1}\big)$. 
 
 Any $n\in\Z^+$ can be decomposed uniquely as $n=s^2m$ with $m$ squarefree and $s\in\Z^+$. 
 A well known result due to Besicovitch states that the square roots of squarefree numbers are linearly independent over $\Q$.
 Hence, for each $\vec{n}$ with $L(\vec{n})= 0$ there are partitions
 \[
  \bigcup_{i=1}^\ell A_i
  =\{1,2,\dots,k\}
  \qquad\text{and}\qquad
  \bigcup_{i=1}^\ell B_i
  =\{k+1,k+2,\dots,2k\}
 \]
 selecting the coordinates with the same squarefree part, which must cancel the squared parts. In formulas, 
 $n_j=s_j^2m_i$ for every $j\in A_i\cup B_i$ with $m_i$ distinct
 and $\sum_{j\in A_i}s_j=\sum_{j\in B_i}s_j$.
 
 The case $\ell=k$ corresponds to $\#A_i=\#B_i=1$, hence 
 $A_i=\{\tau(i)\}$, $B_i=\{k+\lambda(i)\}$
 for some permutations $\tau$ and $\lambda$ of $\{1,2,\dots,k\}$.
 In particular, $n_{\tau(i)}=s_{\tau(i)}^2m_i=n_{k+\lambda(i)}$ and 
 $\{n_1,\dots,n_k\}=\{n_{k+1},\dots,n_{2k}\}$.
 Consequently, if these sets are not equal then $\ell <k$ and the contribution to the sum \eqref{auxm} is bounded by 
 \begin{equation}\label{auxm2}
  \sum_{\ell=1}^{k-1}
  \sum_{\{A_i\}_{i=1}^\ell}
  \sum_{\{B_i\}_{i=1}^\ell}
  \mathop{\sum\ \sum\ \cdots\  \sum}_{m_1<m_2<\cdots <m_\ell\le R^\sigma} 
  \frac{1}{m_1m_2\cdots m_\ell}
  \sum 
  \frac{1}{s_1s_2\cdots s_{2k}}
 \end{equation}
 where 
 in the inner sum we have the restrictions 
 $s_j\le \sqrt{R^\sigma/m_i}$ for $j\in A_i\cup B_i$
 and
 $\sum_{j\in A_i}s_j=\sum_{j\in B_i}s_j$. 
 Let us see that this inner sum is bounded. Say that the largest $s_j$ is $s_{j_1}$ with $j_1\in A_{i_1}$ (the case $j_1\in B_{i_1}$ is symmetric). 
 Then $u= \sum_{j\in A_{i_1}}s_j=\sum_{j\in B_{i_1}}s_j<ks_{j_1}$
 and the same holds replacing $s_{j_1}$ by the greatest $s_j$ with $j\in B_{i_1}$. 
 Obviously, any other variable is at most $u$ and we have 
 \[
  \sum 
  \frac{1}{s_1s_2\cdots s_{2k}}
  <
  k^2
  \sum_{u=1}^\infty
  \frac{1}{u^2}
  \Big(\sum_{s\le u}\frac{1}{s}\Big)^{2k-2}
  \ll
  \sum_{u=1}^\infty
  \frac{(\log u)^{2k-2}}{u^2}
  \ll 1.
 \]
 Then \eqref{auxm2} is $O\big((\log R)^{k-1}\big)$.
\end{proof}

\begin{proof}[Proof of Theorem~\ref{visible}]
 Take $k=\big\lceil p/(2p-2)\big\rceil$.
 By Theorem~\ref{shorter} and H\"older's inequality
 \[
  R^p(\log R)^p
  \ll 
  \Big(\int |g_\sigma E^*|\, d\nu\Big)^p
  \le 
  \int |E^*|^p\, d\nu
  \cdot 
  \Big(\int |g_\sigma |^q\, d\nu\Big)^{p/q}.
 \]
 The last factor is at most $\Big(\int |g_\sigma |^{2k}\, d\nu\Big)^{p/2k}\ll (\log R)^{p/2}$ by Lemma~\ref{moments} with~$\sigma$ small enough.
\end{proof}

Before entering into the proof of Theorem~\ref{shorter}, let us recall something else about the notation and results in \cite{ChCrUb}, namely Lemma~2.1 and Lemma~2.2 there. 
\smallskip 
 
Let $E(R)$ be the lattice point error for the ball i.e.,
$\#\big\{\vec{n}\in\Z^3\,:\,\|\vec{x}\|\le R\big\}-\frac{4}{3}\pi R^3$.
This quantity  is related to $E^*$ through the formula 
\begin{equation}\label{all2vib}
 E^*(t)
 =
 \sum_{d\le 2R}
 \mu(d) E(t/d)
 +
 o(t)
 \qquad\text{for}\quad 1<t<2R.
\end{equation}
A smoothed Voronoi formula for $E(R)$ is (cf. \cite[\S4.4]{IwKo})
\begin{equation}\label{sphf}
 E(t)
 =
 -
 \frac{R}{\pi}
 \sum_{n=1}^\infty
 \frac{a_n}{\sqrt{n}}
 \cos\big(2\pi t\sqrt{n}\big)
 +
 T(t)+U(t)
\end{equation}
where
\[
 a_n
 =
 \frac{r_3(n)}{\sqrt{n}}
 \widehat{\phi}
 \Big(
 \frac{\sqrt{n}}{M}
 \Big),
 \quad 
 M=\frac{R}{(\log R)^{1/3}},
 \quad 
 \phi\in C_0^\infty\big((-1,1)\big) \text{ even, }
 \widehat{\phi}(0)=1
\]
and $T$ and $U$ are less important terms. 

Using \eqref{all2vib} and \eqref{sphf}, when computing $I(R)$ it appears a term of the form 
\begin{equation}\label{maincon}
 -\frac{1}{\pi}
 \sum_{d<2R}
 \sum_{m\le R^\sigma}
 \sum_{n=1}^\infty
 \frac{\mu(d)a_n}{d\sqrt{mn}}
 \int 
 t\cos\big(2\pi t\sqrt{m}\big)
 \cos\big(2\pi \frac{t}{d}\sqrt{n}\big)
 \, d\nu(t).
\end{equation}

In Proposition 3.1, 3.2 and 3.3 of \cite{ChCrUb} it is proved 
that the terms with $\sqrt{m}\ne \sqrt{n}/d$ as well as those coming from the integration of $T$ and $U$ contribute 
$O\big(R(\log R)^{5/6}\big)$
when $R^\sigma$ is replaced by $M^2$ in the definition of $g_\sigma$.
The important point is that these terms are estimated in absolute value and then the bound still holds with our $g_\sigma$ because $\sigma<2$ assures $R^\sigma=o(M^2)$ and there are less terms. 

\begin{proof}[Proof of Theorem~\ref{shorter}]
 After the previous comments, we have to prove that the terms in \eqref{maincon} with $\sqrt{m}=\sqrt{n}/d$ contribute $-\sigma C R\log R$ asymptotically. 
 For them the integral in \eqref{maincon} is 
 \[
  \int
  t
  \cos^2\big(2\pi t\sqrt{m}\big)
  \, d\nu(t)
  =
  \frac{1}{2}R
  \int_1^2 
  t\psi(t)\, dt
  +
  \frac{R}{2}
  \int_1^2 
  t\psi(t)\cos\big(4\pi Rt\sqrt{m}\big)\, dt.
 \]
 By repeated partial integration, the last integral decays faster than any negative power of $R$. Then the result is deduced if we prove 
 \begin{equation}\label{maincon2}
  M_\sigma(R)
   \sim 
  2\pi C_0
  \sigma \log R
  \qquad\text{with}\quad
  M_\sigma(R)
  =
  \sum_{d<2R}
  \mathop{\sum_{m\le R^\sigma}\sum_{n=1}^\infty}_{d\sqrt{m}=\sqrt{n}} 
  \frac{\mu(d)a_n}{d\sqrt{mn}}.
 \end{equation}
 In \cite{ChCrUb} the range of $m$ and the range of $n$ in which $a_n$ is not negligible were balanced and $d$ was essentially only subject to $d^2\mid n$. In this situation, the identity
 $\sum_{d^2\mid n}\mu(d)=\mu^2(n)$ and its positivity play an important role. Now, the ranges of $m$ and $n$ are unbalanced and small values of $d$ are forbidden for $n$ large, ruining the application of the exact identity. It forces to take a roundabout way with similar ingredients. 
 
 Substituting $n=md^2$ and the definition of $a_n$
 \[
  M_\sigma(R)
  =
  \sum_{d<2R}
  \frac{\mu(d)}{d^6}
  \sum_{m\le R^\sigma}
  r_3(md^2)f_d(m)
  \quad\text{with}\quad
  f_d(x)
  =
  \Big(\frac{d}{\sqrt{x}}\Big)^{3}\widehat{\phi}\Big(\frac{d\sqrt{x}}{M}\Big).
 \]
 
 The properties of the Hecke operators give \cite[\S7]{HiSe}, for $d$ squarefree, 
 \[
  r_3(md^2)
  \le 
  r_3(m)
  \prod_{p\mid d}(p+2)
  <
  r_3(m)d
  \prod_{p\mid d}\big( 1+p^{-1}\big)^2
  \ll
  r_3(m)(\log\log d)^2d. 
 \]
 Then the contribution to $M_\sigma(R)$ of $m\le d^2$ is bounded by
 \[
  \sum_{d<2R}
  \frac{1}{d^6}
  \sum_{m\le d^2}
  \frac{r_3(m)(\log\log d)^2d^4}{m^{3/2}}
  \ll 
  \sum_{d<2R}
  \frac{(\log\log d)^2}{d^2}\log d\ll 1
 \]
 because $\sum_{m\le N} r_3(m)\ll N^{3/2}$. 
 Consequently, we can restrict the sum in $M_\sigma(R)$ to $d^2<m\le R^\sigma$ losing $O(1)$.
 
 On the other hand, Gauss' elementary geometric argument to count lattice points \cite{ivicetall} (cf. \cite[\S2]{ChCrUb}) applied to the lattice $(d^2\Z)^3$ shows for $N\ge d^2$  
 \[
  R_d(N)
  =
  \frac{4\pi}{3}
  \Big(\frac{N}{d^2}\Big)^{3/2}
  +
  O\Big(\frac{N}{d^2}\Big)
  \qquad\text{where}\quad
  R_d(N)
  =
  \frac{1}{v(d)}
  \sum_{0\le m\le N}
  r_3(md^2)
 \]
 with $v(d)$ the number of solutions of $x^2+y^2+z^2\equiv 0\mod{d^2}$.
 
 By Abel's summation formula,
 \[
  \sum_{d^2<m\le R^\sigma}
  f_d(m)\frac{r_3(md^2)}{v(d)}
  =
  R_d(R^\sigma)f_d(R^\sigma)
  -
  R_d(d^2)f_d(d^2)
  -
  \int_{d^2}^{R^\sigma}
  R_d(t) f'_d(t)\,dt.
 \]
 It is easy to see $f_d(t),\, tf_d'(t)\ll d^3t^{-3/2}$, because $\widehat{\phi}(z)$ and $z{\widehat{\phi}}'(z)$ are bounded (and rapidly decreasing). It implies that 
 $R_d(t)f_d(t)$ and $d^{-2}\int_{d^2}^{R^\sigma} t\big|f_d'(t)\big|\, dt$ are $O(1)$.
 Recalling the last part of Lemma~\ref{v_eval},
 \[
  M_\sigma(R)
  =
  -\frac{4\pi}{3}\sum_{d<R^{\sigma/2}}
  \frac{\mu(d)v(d)}{d^9}
  \int_{d^2}^{R^\sigma}
  t^{3/2} f'_d(t)\,dt
  +O(1).
 \]
 The range $d<R^{\sigma/2}$ is forced by the previous restriction to $d^2<m\le R^\sigma$.
 Integrating by parts, this equals
 \[
  2\pi\sum_{d<R^{\sigma/2}}
  \frac{\mu(d)v(d)}{d^9}
  \int_{d^2}^{R^\sigma}
  t^{1/2} f_d(t)\,dt
  +O(1).
 \]
 Unwrapping the definition of $f_d(t)$
 and introducing
 \[
  s_d
  =
  2\pi\chi(d)
  \int_{d^2}^{R^\sigma}
  t^{-1}
  \widehat{\phi}\Big(\frac{d\sqrt{t}}{M}\Big)\, dt
  -2\pi\sigma\log R
 \]
 with $\chi$ the characteristic function of $[1,R^{\sigma/2}]$, we have
 \[
  M_\sigma(R)
  =
  2\pi\sigma(\log R)
  \sum_{d=1}^\infty
  \frac{\mu(d)v(d)}{d^6}
  +
  \sum_{d=1}^\infty
  \frac{\mu(d)v(d)}{d^6}s_d
  +O(1).
 \]
 The first sum is, by Lemma~\ref{v_eval},
 \[
  \prod_p
  \Big(1-\frac{v(p)}{p^6}\Big)
  =
  \Big(1-\frac{8}{2^6}\Big)
  \prod_p
  \Big(1-\frac{p^2+p-1}{p^4}\Big)
  =
  C_0.
 \]
 
 It remains to check that the second sum is negligible. Choose $\delta = \frac 14 \min(\sigma,2-\sigma)$.
 Recall that $\widehat{\phi}$ is regular, even and $\widehat{\phi}(0)=1$, so  $\widehat{\phi}(x)=1+O(x^2)$. Substituting this in $s_d$, 
 \[
  s_d\ll \log d+
  \int_{d^2}^{R^\sigma}
  \frac{d^2}{M^2}\,dt
  \ll \log d
  \qquad\text{for $d\le R^\delta$}
 \]
 giving a bounded contribution to the sum. 
 For $d\ge R^\delta$ we use the trivial bound $s_d=O(\log R)$ and Lemma~\ref{v_eval} to get 
 \[
  (\log R)\sum_{d>R^{\delta}}\frac{\log\log d}{d^2}
  \ll 
  \frac{\log\log R}{R^{\delta}}\log R=o(1).
 \]
 Summing up, $M_\sigma(R)=2\pi\sigma C_0\log R+O(1)$, which is a strong form of the required formula \eqref{maincon2}.
\end{proof}
 
\

The resonator giving  Theorem~\ref{gaps} 
is a variant of the Fej\'er kernel and the gaps between the frequencies  assure that we can capture one of them without interferences, obtaining the expected correlation. 
Namely, we choose
\[
  g(x)
  =
  \frac{e(-\nu_n x)}{\Lambda_n}
  \Big(
  \frac{\sin(\pi\Lambda_n x)}{\sin(\pi x)}
  \Big)^2
  =
  \sum_{k=-\Lambda_n}^{\Lambda_n}
  \Big(
  1-\frac{|k|}{\Lambda_n}
  \Big)
  e\big( (k-\nu_n)x\big).
\]
This idea appeared in a different context in  \cite[Prop.\,3.3]{ChUb}. 

\begin{proof}[Proof of Theorem~\ref{gaps}] 
 Let $S$ be the sine sum in the statement.  Since $e(x)+e(-x)=2i\sin(2\pi x)$  and $\Lambda_n\ge |\nu_n-\nu_m|$ for $m\ne n$,
 \[
  \int_{-1/2}^{1/2} S(t)g(t)\, dt = \frac 12 a_n
  \qquad\text{where $g$ is as above}.
 \]
 Hence for $x$ in a positive measure subset of  $\big[-\frac 12,\frac 12\big]$ we have
 \[
  |x|^{-\alpha}\big| S(x)\big| 
  \int_{-1/2}^{1/2} |t|^\alpha |g(t)|\, dt 
  \ge
  \frac 12 |a_n|
 \]
 and the result follows if we check (note that $|g|$ is even)
 \begin{equation}\label{auxgap}
  \mathcal{B}_{n,\alpha}
  \int_{0}^{1/2} t^\alpha \Lambda_n|g(t)|\, dt 
  <
  \Lambda_n.
 \end{equation}
 Let $I_1$ and $I_2$ be the contributions to the integral of $t\le (\pi\Lambda_n)^{-1}$ and $t\ge (\pi\Lambda_n)^{-1}$. In $I_1$ we use the trivial bound 
 $\Lambda_n|g(t)|\le \Lambda_n^2$
 and in $I_2$, 
 $\Lambda_n|g(t)|<\csc^2(\pi t) \le (\pi t)^{-2}+0.6$
 because $\csc^2 t-t^{-2}$ is increasing in $(0,\pi/2]$. Substituting these bounds, 
 \[
  I_1+I_2
  <
  \frac{1}{\pi^{\alpha+1}}
  \Big(
  \frac{2\Lambda^{1-\alpha}}{1-\alpha^2}
  -
  \frac{(2/\pi)^{1-\alpha}}{1-\alpha}
  \Big)
  +
  0.3
  <
  \frac{(\pi\Lambda)^{1-\alpha}}{1-\alpha^2}
  -
  \frac{2^{-\alpha}}{1-\alpha}
  +
  0.3
 \]
 for $\alpha<1$ and the limit $\alpha\to 1^+$ makes sense. 
 The function $f(\alpha)= 2^{-\alpha} +0.3(\alpha-1)-(1+\alpha)^{-1}\big(\frac{\pi}{5}\big)^{1-\alpha}$ is decreasing in $[0,1]$ (a tedious proof consists in subdividing the interval in a number of pieces and use trivial bounds to get $f'<0$ in each of them). Then $f(\alpha)\ge f(1)=0$ and we can add $f(\alpha)/(1-\alpha)$ to the previous bound for $I_1+I_2$ to obtain
 \[
  I_1+I_2
  <
  \frac{\Lambda_n^{1-\alpha}-(1/5)^{1-\alpha}}{(1-\alpha^2)\pi^{\alpha-1}}
 \]
 and \eqref{auxgap} follows. 
\end{proof}

\begin{proof}[Proof of Theorem~\ref{thconv}]
 We take as resonator the shifted Dirichlet kernel 
 \[
  g(t)
  =
  \sum_{n=1}^N e(-nt).
 \]
 Consider a probability measure $d\mu=R^{-1}\psi(x/R)\, dx$
 with $\psi(x)\ne 0$ for $x\in [1,2]$, $\psi\in C_0^\infty\big((1/2,5/2)\big)$. We are going to prove
 \begin{equation}\label{main2}
  \int Gg\, d\mu\sim \mathcal{B}.
 \end{equation}
 Hence Cauchy's inequality and $\int |g|^2\, d\mu \ll N$ imply 
 $\int |G|^2\, d\mu \gg \mathcal{B}^2/N$ giving the result. 
 
 Substituting the definitions of $G$ and $g$ and changing the variable $t=Rx$, the integral \eqref{main2} is 
 \[
  I =
  \sum_{m=1}^N
  \sum_{n=1}^N
  \sum_{d=1}^\infty
  f(d)a_n
  \int \alpha_d(x) 
  e\big( R(n/d-m)x\big)\, dx
 \]
 with $\alpha_d(x)=\phi(Rx/d)\psi(x)$.
 
 The terms with $m=n/d$ contribute 
 \[
  \sum_{n=1}^N
  a_n
  \sum_{d\mid n}
  f(d)
  \int \alpha_d(x) 
  \, dx
  =
  \mathcal{B}
 \]
 because $R\ge 4N$ assures $\alpha_d=\psi$ for $d<R/4$ and $\int\psi=1$. 
 
 On the other hand, if $d>3R$ then $\alpha_d$ is  identically zero. Hence, for any $d$ by partial integration $ \widehat{\alpha}_d(\xi)\ll (1+|\xi|)^{-2K}$. Then the contribution of the terms with $m\ne n/d$ is bounded by 
 \[
  \sum_{d\le 3R}
  \mathop{\sum_{m=1}^N\ \sum_{n=1}^N}_{n/d\ne m} 
  |a_n|\Big(1+R\big|\frac nd-m\big|\Big)^{-2K}
  =
  \sum_{\substack{j=-\infty\\ V=2^j}}^\infty
  \mathop{\sum_{d\le 3R}\ \sum_{m=1}^N\ \sum_{n=1}^N}_{\frac{dV}{R}\le |n- md|<\frac{2dV}{R}} 
  |a_n|(1+V)^{-2K}.
 \]
 Write $r=|n-md|$ then $n\equiv \pm r\pmod{d}$ and given $n$, $r$ and $d$, the values of $m$ remains determined. 
 Note that the sum is empty if $d\le R/(2V)$, which forces $V>1/6$. So, we can write the sum as 
 \[
  \sum_{\substack{V=2^j\\ V>1/6}}^\infty
  \sum_{\frac{R}{2V}\le d\le 3R}\ 
  \sum_{\frac{dV}{R}\le r<\frac{2dV}{R}}
  \sum_{\substack{n=1\\ n\equiv \pm r\ (d)}}^N 
  |a_n|(1+V)^{-2K}.
 \]
 Using that $(6V)^K/(1+V)^{2K}$ is bounded, the assumption in the statement assures that this is $o(\mathcal{B})$ proving \eqref{main2}.
\end{proof}

\subsection*{Acknowledgments}
I am deeply indebted to Professor Iwaniec for his decisive help along my career and for his outstanding mathematical contributions, his books and his beautiful lecture notes that are a gift for the number theoretical community. 
I also want to take the opportunity to thank E.~Valenti.
This work was partially supported by the PID2020-113350GB-I00 grant of the MICINN (Spain),
by ``Severo Ochoa Programme for Centres of Excellence in R{\&}D'' (CEX2019-000904-S)
and by the Madrid Government agreement with UAM in the context of the V PRICIT.
I am grateful to A.~Ubis for pointing out some corrections. 


\normalsize





\normalsize
\baselineskip=17pt




\end{document}